\documentclass[11pt]{article}
\usepackage[numbers]{natbib}
\usepackage{fullpage}
\usepackage{setspace}
\usepackage{amsmath}
\usepackage{amssymb}
\usepackage{mathtools}
\usepackage{amsthm}
\usepackage{relsize}
\usepackage{enumitem}
\usepackage{tabularx}
\usepackage{makecell}
\usepackage{diagbox}
\newtheorem{lemma}{Lemma}
\newtheorem{theorem}{Theorem}
\newtheorem{conjecture}{Conjecture}
\newtheorem{corollary}{Corollary}
\theoremstyle{remark}
\newtheorem{ex}{Example}
\newcommand\NN{\mathbb{N}}
\newcommand\WW{\mathcal{W}}
\newcommand\HH{\mathcal{H}}
\newcommand\II{\mathbb{I}}
\newcommand\RR{\mathbb{R}}

\DeclareMathOperator{\GCD}{GCD}
\DeclareMathOperator{\Hyp}{Hyp}
\begin{document}
\pagenumbering{gobble}
\title{The idealizer of the semigroup of stable polynomials}
\author{\textbf{Michał Kudra}\footnote{Partially supported by Sheng grant no. 2023/48/Q/ST1/00048 of the National Science Center, Poland}\\
\small{Faculty of Mathematics and Computer Science, Jagiellonian University}\\
\small{Łojasiewicza 6, 30-348, Kraków, Poland}\\
\small{$\&$}\\
\small{Doctoral School of Exact and Natural Sciences}\\
\small{Jagiellonian University}\\
\small{Łojasiewicza 11, 30-348, Kraków, Poland}\\
\small{\texttt{michal.kudra@doctoral.uj.edu.pl}}\\
}
\maketitle

\begin{abstract}
It follows from the Garloff–Wagner Theorem that the set of stable polynomials of degree~$n$, denoted by $\mathcal{H}_n$, i.e., those whose all zeros lie in the left open complex half-plane, with the Hadamard product *, forms an abelian semigroup contained in the abelian group $\mathbb{R}_n^+$ of polynomials of degree $n$ with positive real coefficients. By the idealizer of the set $\mathcal{H}_n$, we refer to the largest subsemigroup of $\mathbb{R}_n^+$ in which $\mathcal{H}_n$ is an ideal. In this paper, we formulate a conjecture characterizing the idealizer of $\HH_n$ and prove it for $n \leqslant 5$. In addition, we show that the proposed condition is necessary for any polynomial to belong to the idealizer and establish, in a~distinguished special case, a sufficient condition of a similar nature that supports the conjecture.
\end{abstract}


\section{Introduction}

Let $\RR_n^+$ denote polynomials of degree $n$ with positive coefficients. A non-constant polynomial $f$ of degree $n$ with real coefficients is called \textbf{stable} if all its zeros lie in the open left half of the complex plane and \textbf{quasi-stable} if all its zeros lie in the closed left half of the complex plane. One of the necessary conditions for stability is that the coefficients of a given polynomial are of the same sign, therefore, for simplicity, when discussing stable polynomials, we only consider the set
\begin{equation*}
\HH_n = \left\{ f \in \RR_n^+ : f(x)=0 \Rightarrow \Re(x)<0 \right\}.
\end{equation*}
In the case of quasi-stability, we define the set
\begin{equation*}
\HH_n^\star = \left\{ f(x)=\sum_{i=0}^na_ix^i,\; a_i \geqslant 0 \;
(i=1,\ldots,n-1),\; a_0,a_n>0 : f(x)=0 \Rightarrow \Re(x)\leqslant 0
\right\},
\end{equation*}
which for the sake of our work represents all quasi-stable polynomials as the assumption $a_0>0$ does not affect the generality of our reasoning. Details are explained in Subsection \ref{had}.

It follows from the Garloff–Wagner Theorem \cite[Thm.~1]{gw1996} that $\HH_n$ with the coefficient-wise product $\ast$, so called the Hadamard product (see Subsection \ref{had}), is an abelian semigroup in the abelian group $\RR_n^+$. Only for $n=1,2$ the semigroup $\HH_n$ is a group, since $\HH_1=\RR_1^+$ and $\HH_2=\RR_2^+$. For $n=3$, we see that $\HH_3$ does not contain the identity polynomial $\II_3$ with respect to the product~$\ast$,
i.e., $\II_n(x):=\sum_{j=0}^n x^j$. We must therefore  broaden our perspective to include $\HH_3^\star$, which from the Hermite-Biehler Theorem, is a subset of $\RR_3^+$. It does indeed contain $\II_3$, but it contains no elements that are inverses of stable polynomials (by the Routh–Hurwitz Theorem). That means it is a~commutative monoid, but not a group. However, $\HH_4^\star$ cannot be a monoid, since $\II_4 \notin \HH_4^\star$.

For each $n$, we seek a maximal subsemigroup $X_n \subset \RR_n^+$ with the Hadamard product $\ast$, such that $\HH_n \subset X_n$, $\II_n \in X_n$, and for $m \leqslant\ n$
\begin{equation}\label{ideal}
\ast : X_n \times \HH_m \rightarrow \HH_m.
\end{equation}
In other words, we want to characterize
\begin{equation*}
X_n = \{ g \in \RR_n^+ : \forall m \leqslant n, \; \forall f \in \HH_m \;\;\;
f \ast g \in \HH_m \} = \{g \in \RR_n^+ : \forall m \leqslant n \;\;\;
g \ast \HH_m \subset \HH_m\}.
\end{equation*}
The set $X_n$ defined this way is the largest submonoid of the group $\RR_n^+$, in which $\HH_m$ is an ideal. Such a set is called the \textbf{idealizer} of $\HH_m$ with respect to the product $\ast$.

If $n=1,2$, then $X_n=\RR_n^+$. It is easy to see that for $n=3$ we have
$X_3 = \HH_3^\star$. As for $n=4$, in the paper \cite{bg2021}, a~family $W_n$
was introduced for which condition (\ref{ideal}) holds in this case. We show
that the desired maximal subsemigroup for $n=4$ is $X_4 = \overline{\WW_4}$,
where $\overline{\WW_n}$ denotes the closure of the set $\WW_n$
(see Section \ref{main}). To our knowledge, no general approach for addressing
related problems has yet been established. Available methods seem largely limited
in scope, and as some recent papers show, several results have been obtained only
for polynomials of degree $n \leqslant 5$ or less (see \cite{aag2024},
\cite{bbc2017}, \cite{bbck2024}, \cite{bg2021}, \cite{wz2014}). After
considering the case where $n=4$, for $n \geq 5$, one can suppose that
$\overline{\WW_n}$ is a good candidate for the maximal subsemigroup $X_n$.
However, in Example \ref{ex1}, we show that $\overline{\WW}_5 \neq X_5$.
In Section \ref{necessary}, we give a definition of a family of polynomials
$Y_n$ for which $Y_1=\RR_1^+=X_1$, $Y_2=\RR_2^+=X_2$, $Y_3=\HH_3^\star=X_3$,
and $Y_4=\overline{W_4}=X_4$ (see Theorem \ref{main4}). Therefore, we propose
the following conjecture.
\begin{conjecture}\label{conjecture1}
For any $n \in \NN$, equality $Y_n = X_n$ holds.
\end{conjecture}
The statement is trivial for $n=1,2$, while the case $n=3$ is readily verified.
We prove that Conjecture \ref{conjecture1} is true for $n = 4$ (Theorem \ref{main4}) and $n=5$ (Theorem \ref{main5}). Moreover, we show that $X_n \subset Y_n$ for all $n$ (Theorem \ref{necessarythm}). We also investigate the idealizer of the set of quasi-stable polynomials. In addition, for a special case, we show a sufficient condition that coincides with Conjecture \ref{conjecture1} (Theorem \ref{scase}).

The paper is organized as follows. In Section \ref{basic}, we review the
required background from stability theory. We then show necessary conditions for
belonging to the idealizer $X_n$ in Section \ref{necessary}. The main results
of the paper are Theorems \ref{main4} and \ref{main5}, presented in
Section \ref{main}, where we prove that Conjecture~\ref{conjecture1} is true
for $n=4$ and $n=5$. Section \ref{quasi} provides additional information on the
idealizer of quasi-stable polynomials, and Section \ref{misc} is devoted to
supplementary results in a special case, in support of Conjecture \ref{conjecture1}. Finally, Section \ref{exes} provides examples completing our results.

\section{Basic definitions and theorems}\label{basic}

In the following subsections, we present several characterizations of stability
and quasi-stability, and we recall the classical Garloff–Wagner Theorem. These
results are extensively used throughout the rest of this paper.

\subsection{Stable and quasi-stable polynomials}

The \textbf{Hurwitz matrix} of the polynomial
$f(x)=a_nx^n+\ldots+a_1x+a_0$ is defined as
\begin{equation*}
H(f)={\begin{bmatrix}
a_{n-1} &a_{n-3} &a_{n-5}    &\ldots &0\\
a_n   &a_{n-2} &a_{n-4}  &\ldots &0\\
0     &a_{n-1} &a_{n-3}  &\ldots &0\\
0     &a_{n} &a_{n-2}  &\ldots &0\\
\vdots&\vdots&\vdots &\ddots &\vdots  \\
0     &0     &0      &\ldots &a_0
\end{bmatrix}},
\end{equation*}
and we denote its principal minors by $\Delta_k(f)$. In particular,
$\Delta_1(f)= a_{n-1}$, $\Delta_n(f)=\det{H(f)}=a_0 \Delta_{n-1}(f)$.

Among the best-known equivalent conditions for stability are the Routh–Hurwitz
and Liénard–Chipart Theorems
(see \cite[Thm.~11.4.5, Thm.~11.4.7]{rs2002}).

\begin{theorem}[Routh-Hurwitz]
A polynomial $f \in \mathbb{R}^+_n$ is stable if and only if
\begin{equation*}
\Delta_k(f) >0 \; \textrm{ for all } \; k =1, \ldots, n.
\end{equation*}
\end{theorem}

\begin{theorem}[Liénard-Chipart]
Given a polynomial $f \in \RR_n^+$, the following conditions are equivalent:
\begin{itemize}
\item[(i)]  $f \in \HH_n$,
\item[(ii)] $\Delta_2(f)>0 , \Delta_4(f)>0 , \ldots ,\Delta_{2\lfloor n/2 \rfloor}(f) >0,$
\item[(iii)] $\Delta_1(f)>0 , \Delta_3(f) >0 , \ldots , \Delta_{2\lfloor (n+1)/2 \rfloor -1 } (f)>0,$
\end{itemize}
where $\lfloor x \rfloor$ denotes the integer part of the number $x$.
\end{theorem}

In the case of quasi-stable polynomials, it is easy to obtain weak inequalities
for the principal minors as the necessary conditions. However, these are not
equivalent. They were first formulated in Theorem 4.9 in \cite{agt2018}. By
$f_e$ and $f_o$ let us denote even and odd parts of a polynomial $f$, such that
$f(x) = f_e(x^2) + xf_o(x^2)$.

\begin{theorem}[Adm, Garloff, Tyaglov]\label{agt}
Given polynomial $f$ with real coefficients, the following statements are
equivalent:
\begin{itemize}
\item[(i)] The polynomial $f$ belongs to $\HH_n^\star$ with the stability index m, that is, $f$ has $m$ zeros in the open left complex half-plane.
\item[(ii)] The principal minors of the Hurwitz matrix of $f$ are positive up to order $m$, that is
\begin{equation*}
\Delta_1(f)>0,\;\Delta_2(f)>0,\;\ldots \Delta_m(f)>0 ,\;
  \Delta_{m+1}(f)=\ldots =\Delta_n(f)=0
  \end{equation*}
and $\GCD(f_e,f_o)$ has only negative zeros, where $\GCD(f,g)$ denotes a monic polynomial of the highest possible degree simultaneously dividing polynomials
$f$ and $g$.
\end{itemize}
\end{theorem}
This theorem plays a crucial role in the proofs of Theorem \ref{sufficient4bg} and Lemmas \ref{rownowaznosc1} and \ref{rownowaznosc2}.

\subsection{Interlacing and Hermite-Biehler Theorem}

Suppose that $h \in \RR_m^+$ and $g \in \RR_n^+$ are polynomials with only real zeros, where those of $h$ are $\alpha_1,\ldots, \alpha_m$, and those of $g$ are $\beta_1,\ldots,\beta_n$. We say that $g$ \textbf{interlaces} $h$ if either $m=n+1$, and
\begin{equation*}
\alpha_1 \leqslant \beta_1 \leqslant \alpha_2 \leqslant \ldots
  \leqslant \beta_n \leqslant \alpha_{n+1}
\end{equation*}
or $m=n$, and
\begin{equation*}
\beta_1 \leqslant \alpha_1 \leqslant \beta_2 \leqslant \ldots
  \leqslant \beta_n \leqslant \alpha_n.
\end{equation*}
We denote the interlacing property by $g \prec h$. Moreover, for any polynomial $f \in \RR_n^+$ with only real zeros, we adopt the convention that $0 \prec f$ and $f \prec 0$. We are now ready to formulate the following theorem, which is a slightly simplified version of the Hermite–Biehler Theorem given by Garloff and Wagner, adapted to polynomials from $\HH_n^\star$ (cf.\ \cite[Thm.~3]{gw1996}).

\begin{theorem}[Hermite-Biehler]\label{hb}
Let $f$ be a polynomial. Then
\begin{itemize}
\item[(i)]  $f \in \HH_n^\star \Leftrightarrow f_e$ and $f_o$ have only negative zeros and $f_o \prec f_e$,
\item[(ii)]  $f \in \HH_n \Leftrightarrow f \in \HH_n^\star$ and $\GCD(f_e, f_o) = 1$,
\item[(iii)] $f$ has zeros only on the imaginary axis $\Leftrightarrow f \in \HH_n^\star$ and $f_o=0$,
\item[(iv)]  $f$ has exactly one zero on the negative half-axis and the reston the imaginary axis $\Leftrightarrow f \in~\HH_n^\star,\; f_o \neq 0$, and $f_e =c f_o$.
\end{itemize}
\end{theorem}

The above description of stability and quasi-stability lies at the origin of the family $Y_n$ defined in Section \ref{necessary}.

\subsection{Hadamard product}\label{had}

\textbf{Hadamard product} of the polynomials
\begin{equation*}
f(x)=a_nx^n+a_{n-1}x^{n-1}+ \ldots + a_1 x +a_0,
\end{equation*}
\begin{equation*}
g(x)=b_mx^m+b_{m-1}x^{m-1}+ \ldots + b_1 x +b_0,
\end{equation*}
is defined to be
\begin{equation*}
(f \ast g) (x) := a_kb_kx^k + a_{k-1}b_{k-1}x^{k-1} + \ldots +
a_1b_1x+a_0b_0, \textrm{ where } k=\min\{n,m\}.
\end{equation*}

The problem of determining the zeros of $f*g$ has been studied by many mathematicians, including Maló, Weisner, Schur, de Bruijn, and Takagi (see \cite[Sec.~5]{rs2002}). However, the result most relevant to us comes from Garloff and Wagner. As the Hadamard product with quasi-stable polynomials requires a bit more careful treatment, we present a slightly simplified version of their theorem (cf.\ \cite[Thm.~1]{gw1996}).

\begin{theorem}[Garloff-Wagner]
Let $f \in \HH_n^\star$, $p \in \HH_m^\star$, and $k= \min \{m,n\}$.
\begin{itemize}
\item[(i)] Then $f \ast p \in \HH_k^\star$.
\item[(ii)] If $f$ or $p$ has zeros exclusively on the imaginary axis, then so does $f \ast p$.
\item[(iii)] If both $f$ and $p$ have exactly one zero on the negative half-axis and the rest on the imaginary axis, then so does $f * p$.
\item[(iv)]  If neither $(ii)$ nor $(iii)$ holds, then $f \ast p$ has no zeros on the imaginary axis.
\item[(v)] As a \textbf{special case of $(iv)$}, if $f \in \HH_n$ and $p \in \HH_m$, then $f \ast p \in \HH_k$.
\end{itemize}
\end{theorem}

The table below illustrates all possible cases:
\begin{center}
\setlength{\tabcolsep}{1pt}
\begin{tabularx}{\textwidth}{|l|X|X|X|X|}
\hline
\diagbox{$f \in \HH_n^\ast$}{$f \in \HH_m^\ast$}
& $p_o = 0$
& $p_e = bp_o \neq 0$
& \makecell[l]{
$p_e \neq b p_o$,\\
$p_o \neq 0$
}
& $p \in \HH_m$
\\ \hline
$f_o = 0$
&$f_o \ast p_o =0$
&$f_o \ast p_o =0$
&$f_o \ast p_o =0$
&$f_o \ast p_o =0$ \\
\hline
$f_e = a f_o \neq 0$
&$f_o \ast p_o =0$
&$f_e \ast p_e = ab f_o \ast p_o$
&$\HH_k$
&$\HH_k$
\\ \hline
\makecell[l]{
$f_e \neq a f_o$,\\
$f_o \neq 0$
}
&$f_o \ast p_o =0$
&$\HH_k$
&$\HH_k$
&$\HH_k$
\\ \hline
$f \in \HH_n$
&$f_o \ast p_o =0$
&$\HH_k$
&$\HH_k$
& $\HH_k$ \\ \hline
\end{tabularx}
\end{center}
where $a,b>0$.

Observe that the Hadamard product of the quasi-stable polynomial
$f(x)=x^6+x^4=x^4(x^2+1)$ with a polynomial $g(x)=b_6x^6+\ldots+ b_1 x + b_0$ is given by
$(f*g)(x)=b_6x^6+b_4x^4.$ Quasi-stability of this product is equivalent to the quasi-stability of the Hadamard product of the polynomials $p(x)=x^2+1$ and $q(x)=b_6x^2+b_4$.
Therefore, throughout this paper, without loss of generality, we restrict our
attention to quasi-stable polynomials with a positive constant term, as
specified in the definition of $\HH_n^\star$ in the Introduction.

\section{Necessary conditions}\label{necessary}

When beginning research on any mathematical object, it is often reasonable to
consider its behavior in extreme cases. That is an exact idea behind the
Pólya-Schur Theorem \cite[Thm.~5.7.2]{rs2002}, where sequences that are
supposed to preserve the reality of polynomial zeros (sometimes called
\textbf{hyperbolicity}) are tested solely on polynomials of the form $(x+1)^n$,
whose zeros all coincide at the single point. Such so-called \textbf{multiplier
sequences} continue to attract interest (see \cite{cpk2011}, \cite{kv2024},
\cite{k2008}). More generally, there is ongoing research on operators preserving
not only the reality of the zeros but also the sectors or strips in which the
zeros of a polynomial lie. These operators are referred to as zero-strip-preserving
(or reducing) operators or zero-sector-preserving (or reducing) operators (see
\cite{c2015}, \cite{cfpsw2018}, \cite{ch2011}). Somewhat surprisingly, these
testing polynomials have remained essentially unchallenged. In our setting, in
addition to preserving the reality of zeros, we must also account for the
interlacing property of the even and odd parts of the polynomial. This
additional requirement naturally leads us to introduce the following class of
objects that are used throughout this paper.

Polynomials of the form
\begin{equation*}
Q^k(x)=(x^2+1)^l+x(x^2+1)^l \qquad \textrm{ for } k=2l+1,
\end{equation*}
\begin{equation*}
Q^k(x)=(x^2+1)^l \qquad \textrm{ for } k=2l,
\end{equation*}
are called \textbf{basic quasi-stable polynomials of degree $k$}, and
\begin{equation*}
Q_m^k(x)=x^mQ^k(x) \qquad  \textrm{ for }  k,m \in \NN, 
\end{equation*}
the \textbf{$m$-shifted basic quasi-stable polynomials of degree $k$}.

Note that the polynomials $Q^k$ defined above are special cases of $(iv)$ and
$(iii)$ in the Hermite–Biehler Theorem. Additionally, we define the family of
polynomials $Y_n$, intended to be the idealizer of the set $\HH_n$, as follows
\begin{equation*}
Y_n:= \left\{ g \in \RR_n^+ : \frac{g*Q_m^k}{x^m} \in \HH_k^\star \;\;
\forall k,m \in \NN, \textrm{ such that } k\geqslant 2, k+m \leqslant n
\right\}.
\end{equation*}
In other words, the family $Y_n$ consists of polynomials such that the Hadamard
product of their every k-term sum of consecutive monomials with the basic
quasi-stable polynomial belongs to $\HH_k^\star$, i.e., polynomials which
preserve the quasi-stability of every possible $Q^k_m$. We now show that the
set $Y_n$ satisfies the necessary condition for belonging to $X_n$.
\begin{theorem}\label{necessarythm}
Suppose that $m \leqslant n,\;g \in \RR_n^+$. The following implications hold
\begin{itemize}
\item[(i)]$f \ast g \in \HH_m \quad \forall f \in \HH_m \Rightarrow g \in Y_n,$
\item[(ii)]$f \ast g \in \HH_m^\star \quad \forall f \in \HH_m^\star \Rightarrow g \in Y_n.$
\end{itemize}
Consequently, for every $n \in \NN$, the idealizer $X_n$ is included in $Y_n$.
\end{theorem}
\begin{proof}
Let us set $\varepsilon,\mu,\alpha,\beta>0$ and construct
\begin{multline*}
Q_{\varepsilon,\mu,\alpha,\beta}(x)=\alpha\prod_{j=0}^{n_1} (j\mu x^2+1)^j
\prod_{j=0}^{n_2} (x^2+1+j\varepsilon)^j\prod_{j=0}^{n_3}(x^2+j\varepsilon)^j+\\
+\beta x\prod_{j=0}^{n_1} (j\varepsilon x^2+1)^j
\prod_{j=0}^{n_2}(x^2+1+j\mu)^j\prod_{j=0}^{n_3}(x^2+j\mu)^j
\end{multline*}
which for $\mu > \varepsilon > 0$, and by choosing appropriate $n_1, n_2, n_3$,
such that $n_1 + n_2 + n_3 = n$, are stable polynomials of degree $n$. By
letting $\varepsilon$ and $\mu$ tend to zero, they converge to any polynomial
$Q_m^k$ of degree $k+m \leqslant n$. To obtain those of even degree with the
appropriate shift, we take the limit as $\alpha$ or $\beta$ tends to zero.
Since we are looking for polynomials in $\RR_n^+$ such that their Hadamard
product with any stable polynomial is stable, it follows in particular that
their Hadamard product with any polynomial $Q_{\varepsilon,\mu,\alpha,\beta}$
must be stable. By the continuous dependence of the zeros on the coefficients,
we obtain that they too must belong to $Y_n$, which proves (i). For (ii), we
use the same argument, as $Q_{\varepsilon,\mu,\alpha,\beta}(x)$ also belongs
to $\HH_m^\star$.
\end{proof}
For lower degrees, the family $Y_n$ can be written in a much simpler form.
It is easy to verify that $Y_j=\RR_j^+$ for $j \leqslant 2$ and
$Y_3=\left\{g \in \RR_3^+ : g*Q^3 \in \HH_3^\star \right\}=\HH_3^\star$.
For $n=4$, let $g \in~\RR_4^+$,
\begin{equation*}
g(x)=b_4x^4 + b_3x^3 + b_2x^2 + b_1x + b_0.
\end{equation*}
The conditions $g*Q^2 \in \HH_2^\star, \;\frac{g*Q^2_1} {x} \in \HH_2^\star,\;
\frac{g*Q^2_2}{x^2} \in \HH_2^\star$ are satisfied, since they are second-degree
polynomials with nonnegative coefficients. The remaining conditions:
$g*Q^3 \in \HH_3^\star,\; \frac{g*Q^3_1}{x} \in\HH_3^\star,\;
g*Q^4\in \HH_4^\star,$ by Theorem \ref{agt} and the Hermite-Biehler Theorem,
are equivalent to the inequalities $b_1b_2-b_0b_3\geqslant0, \;
b_2b_3-b_1b_4\geqslant0, \;b_2^2-b_0b_4\geqslant0.$ We see that the last inequality follows from the previous ones. Hence
\begin{equation*}
Y_4=\left\{g \in \RR_4^+ : \frac{g\ast Q_i^3}{x^i} \in \HH_3^\star
\textrm{ for } i=0,1\right\}=\left\{g \in \RR_4^+,\;g(x)=\sum_{i=0}^4b_ix^i :
b_1b_2 \geqslant b_0b_3,\; b_2b_3 \geqslant b_1b_4\right\}.
\end{equation*}
We thus have the following conclusion.
\begin{corollary}\label{necessary4}
Suppose that $m \leqslant 4$, $g \in \RR_4^+$. The following implications hold
\begin{itemize}
\item[(i)]$f \ast g \in \HH_m \quad \forall f \in \HH_m \Rightarrow \frac{g\ast Q_i^3}{x^i} \in \HH_3^\star \textrm{ for } i=0,1,$
\item[(ii)]$f \ast g \in \HH_m^\star \quad \forall f \in\HH_m^\star \Rightarrow \frac{g\ast Q_i^3}{x^i} \in \HH_3^\star \textrm{ for } i=0,1.$
\end{itemize}
\end{corollary}
To give a simpler version of $Y_5$, let us consider $g \in \RR_5^+,$ where
\begin{equation*}
g(x) = b_5x^5 + b_4x^4 + b_3x^3 + b_2x^2 + b_1x + b_0.
\end{equation*}
The condition $g*Q^5 \in \HH_5^\star$ ensures that all minors of the matrix
\begin{equation*}
H(g*Q^5)=\begin{bmatrix}
b_4 & 2b_2 & b_0 &0 &0 \\
b_5 & 2b_3 & b_1 &0 &0 \\
0& b_4 & 2b_2 & b_0 &0  \\
0& b_5 & 2b_3 & b_1 &0  \\
0& 0& b_4 & 2b_2 & b_0 \\
\end{bmatrix}
\end{equation*}
are nonnegative, as it is known that the Hurwitz matrix of stable polynomials
is totally nonnegative (see \cite{k1982}, as we approach a quasi-stable
polynomial in the limit, the inequalities hold). In particular, we obtain the
inequalities $b_1b_2-b_0b_3 \geqslant 0$ and $b_3b_4-b_2b_5 \geqslant 0,$
which are equivalent to the conditions $g*Q^3\in \HH_3^\star$ and
$\frac{g*Q^3_2}{x^2} \in \HH_3^\star$ by Theorem \ref{agt}. Given the fact that $g*Q^5(x)=g*Q^4(x)+g*Q^4_1(x)$, from the Hermite-Biehler Theorem we
conclude that $g*Q^4 \in \HH_4^\star$ and $\frac{g*Q^4_1}{x} \in \HH_4^\star$.
The only condition we are not able to obtain from quasi-stability of $g*Q^5$ is
$\frac{g*Q^3_1}{x} \in \HH_3^\star$. Hence
\begin{equation*}
Y_5=\left\{g \in \RR_5^+ : g * Q^5 \in \HH_5^\star,\;
\frac{g\ast Q_1^3}{x} \in \HH_3^\star \right\}.
\end{equation*}
\begin{corollary}\label{necessary5}
Suppose that $m \leqslant 5$, $g \in \RR_5^+$. The following implications hold
\begin{itemize}
\item[(i)]$f \ast g \in \mathcal{H}_m \quad \forall f \in \mathcal{H}_m \Rightarrow  g*Q^5 \in \HH_5^\star,\; \frac{g\ast Q_1^3}{x} \in \HH_3^\star,$
\item[(ii)]$f \ast g \in \HH_m^\star \quad \forall f \in \HH_m^\star \Rightarrow g*Q^5 \in \HH_5^\star,\; \frac{g\ast Q_1^3}{x} \in \HH_3^\star.$
\end{itemize}
\end{corollary}

\section{Main results}\label{main}

In the paper \cite{bg2021}, for $n \geqslant 3$, the authors define the set
\begin{equation*}
\WW_n= \left\{ g(x)=\sum_{i=0}^nb_ix^i:
b_i b_{i-1} > b_{i-2} b_{i+1} \; (i=2,\ldots,n-1) \right\},
\end{equation*}
which determines the sufficient conditions for the so-called \textbf{generalized
Hadamard product} of a polynomial of degree $n \geqslant 4$ with any stable
polynomial of degree $m \leqslant 4$ to be stable. As a corollary, we present
a special case of this theorem in which we consider the classical Hadamard
product of polynomials of degree $4$ and $m$
(cf.\ \cite[Thm.~3.1]{bg2021}).

\begin{corollary}[Białas-Góra]
Suppose that $m \leqslant 4$. The following implication holds
\begin{equation*}
g \in \WW_4 \Rightarrow f \ast g \in \HH_m \qquad \forall f \in \HH_m.
\end{equation*}
\end{corollary}

However, the above implication cannot be reversed. No matter how we choose
$f \in \HH_4$, for $g(x)=x^4+x^3+x^2+x+1$, their product $f*g \in \HH_4$,
but $g \notin \WW_4$. Therefore, we need another family of polynomials so that
the estimates used in the proof can be slightly improved. For $n\geqslant 3$,
let us consider the closure of $\WW_n$, i.e.,
\begin{equation*}
\overline{\WW_n}= \left\{g \in \RR_4^+,\;g(x)=\sum_{i=0}^nb_ix^i:
b_ib_{i-1} \geqslant b_{i-2} b_{i+1} \; (i=2,\ldots,n-1) \right\},
\end{equation*}
thanks to which we are able to obtain the following result.

\begin{theorem}\label{sufficient4bg}
Suppose that $m \leqslant 4$. The following implications hold
\begin{itemize}
\item[(i)]$g \in \overline{\WW_4} \Rightarrow f \ast g \in \HH_m \quad \forall f \in \HH_m,$
\item[(ii)]$g \in \overline{\WW_4} \Rightarrow f \ast g \in \mathcal{H}_m^\star \quad \forall f \in\HH_m^\star,$
\end{itemize}
therefore, $\overline{\WW_4}$ is included in $X_4$.
\end{theorem}

\begin{proof}
If $m=1,2$, then the implication is obvious. For $m = 3$, it is an immediate
consequence of the Garloff-Wagner Theorem, as all of the polynomials from
$\overline{W}_3$ are quasi-stable. Hence, we consider only $m=4$. Let
$f(x)=a_4x^4+a_3x^3+a_2x^2+a_1x+a_0$ and
$g(x)=b_4x^4+b_3x^3+b_2x^2+b_1x+b_0$ be such, that $f \in \HH_4^\star$,
$g\in \overline{\WW_4}$. Let us examine the following cases:
\begin{itemize}
\item[(1)] If the stability index of $f$ is 4, then $f \in \HH_4$, and by the Liénard–Chipart criterion, it suffices to show that
\begin{equation*}
\Delta_3(f*g)=\det\begin{bmatrix}
a_3b_3 & a_1b_1 & 0 \\
a_4b_4 & a_2b_2 & a_0b_0 \\
0 & a_3b_3 & a_1b_1
\end{bmatrix}>0.
\end{equation*}
Hence
\begin{multline*}
\det\begin{bmatrix}
a_3b_3 & a_1b_1 & 0 \\
a_4b_4 & a_2b_2 & a_0b_0 \\
0 & a_3b_3 & a_1b_1
\end{bmatrix}=  a_1a_2a_3b_1b_2b_3-a_0a_3^2b_0b_3^2 - a_1^2a_4b_1^2b_4
\geqslant\\ \geqslant b_1b_2b_3(a_1a_2a_3-a_0a_3^2-a_1^2a_4)
= b_1b_2b_3 \Delta_3(f)> 0,
\end{multline*}
which proves (i).
\item[(2)]If the stability index of $f$ is 2, then by Theorem \ref{agt} it
suffices to show that
\begin{equation*}
\Delta_2(f*g)=\det\begin{bmatrix}
a_3b_3 & a_1b_1 \\
a_4b_4 & a_2b_2 \\
\end{bmatrix} > 0,
\end{equation*}
\begin{equation*}
\Delta_3(f*g)=\det\begin{bmatrix}
a_3b_3 & a_1b_1 & 0 \\
a_4b_4 & a_2b_2 & a_0b_0 \\
0 & a_3b_3 & a_1b_1
\end{bmatrix} \geqslant 0,
\end{equation*}
and if $\Delta_3(f*g)=0$, then additionally that
$\GCD((f*g)_e,(f*g)_o)$ has only negative zeros.
So we have
\begin{equation*}
\det\begin{bmatrix}
a_3b_3 & a_1b_1 \\
a_4b_4 & a_2b_2 \\
\end{bmatrix} = a_2a_3b_2b_3 -a_1a_4b_1b_4\geqslant b_2b_3 \Delta_2(f) > 0
\end{equation*}
\begin{multline*}
\det\begin{bmatrix}
a_3b_3 & a_1b_1 & 0 \\
a_4b_4 & a_2b_2 & a_0b_0 \\
0 & a_3b_3 & a_1b_1
\end{bmatrix}=  a_1a_2a_3b_1b_2b_3-a_0a_3^2b_0b_3^2 - a_1^2a_4b_1^2b_4
\geqslant\\ \geqslant b_1b_2b_3(a_1a_2a_3-a_0a_3^2-a_1^2a_4)
= b_1b_2b_3 \Delta_3(f)\geqslant 0
\end{multline*}
if the equality holds, as
\begin{equation*}
(f*g)_o(x_0)=0 \Leftrightarrow x_0=-\frac{a_1b_1}{a_3b_3}
\end{equation*}
and
\begin{equation*}
(f*g)_e(x_0)=-\frac{\Delta_3(f*g)}{a_3^2b_3^2}=0
\end{equation*}
we have that $\GCD((f*g)_e,(f*g)_o)= x+\frac{a_1b_1}{a_3b_3}$.
\item[(3)]If the stability index of $f$ is 0, then by the Hermite–Biehler
Theorem, $a_3 = a_1 = 0$, so $(f*g)(x)=a_4b_4x^4+a_2b_2x^2+a_0b_0$. As
$b_2^2 \geqslant b_0 b_4$, then $a_2b_2^2 \geqslant 4a_0b_4b_0 b_4$ and
$(f*g)_e$ has only negative zeros which proves (ii).
\end{itemize}
\end{proof}

In a private correspondence with the authors of \cite{bg2021}, Góra inquired
whether if the set $\WW_5$ might also determine sufficient conditions for
polynomials of degree 5. However, as Example \ref{ex1} (see Section \ref{exes})
shows, the implication
\begin{equation*}
g \in \WW_5 \Rightarrow f \ast g \in \HH_5 \qquad \forall f \in \HH_5
\end{equation*}
is not true.

Nevertheless, as we have shown in the previous section, the inequalities
satisfied by the polynomials belonging to $\overline{\WW_4}$ coincide with the
necessary conditions from Corollary \ref{necessary4}. We can therefore combine
it with Theorem \ref{sufficient4bg} and obtain the following.

\begin{theorem}\label{main4}
Suppose that $m \leqslant 4,\; g \in \RR_4^+$. The following equivalences hold
\begin{itemize}
\item[(i)]$ f \ast g \in \HH_m \quad \forall f \in \HH_m \Leftrightarrow \frac{g\ast Q_i^3}{x^i} \in \HH_3^\star \textrm{ for } i=0,1,$
\item[(ii)]$ f \ast g \in \HH_m^\star \quad \forall f \in \HH_m^\star  \Leftrightarrow \frac{g\ast Q_i^3}{x^i} \in \HH_3^\star \textrm{ for } i=0,1.$
\end{itemize}
Hence, $X_4=Y_4=\overline{\WW_4}$.
\end{theorem}

In this paper, we extend this result and show that the necessary conditions we
defined in Section \ref{necessary} are also sufficient for $n=5$.

\begin{theorem}\label{main5}
Suppose that $m \leqslant 5,\;g \in \RR_5^+$. The following equivalences hold
\begin{itemize}
\item[(i)]$ f \ast g \in \HH_m \quad \forall f \in \HH_m \Leftrightarrow  g*Q^5 \in \HH_5^\star,\; \frac{g\ast Q_1^3}{x} \in \HH_3^\star,$
\item[(ii)]$ f \ast g \in \HH_m^\star \quad \forall f \in \HH^\star_m \Leftrightarrow g*Q^5 \in \HH_5^\star,\; \frac{g\ast Q_1^3}{x}\in \HH_3^\star.$
\end{itemize}
Thus, $X_5=Y_5$.
\end{theorem}

To prove the above statement, we need several lemmas. Let us consider
\begin{equation}\label{ff}
f(x)= a_5x^5 + a_4x^4 + a_3x^3 + a_2x^2 + a_1x + a_0
\end{equation}
\begin{equation}\label{gg}
g(x)= b_5x^5 + b_4x^4 + b_3x^3 + b_2x^2 + b_1x + b_0
\end{equation}

\begin{lemma}\label{rownowaznosc1}
For $f \in \RR_5^+$ as in \eqref{ff} let us define
\[ A:=\frac{a_1 a_4}{a_2 a_3},\ \ \ \  B:=\frac{a_1 a_5}{a_3^2},
   \ \ \ \  C:=\frac{a_0 a_4}{a_2^2}.\]
Then the following conditions are equivalent:
\begin{itemize}
\item[(i)] $f \in \HH_5^\star$,
\item[(ii)] $\Delta_2(f)=a_3a_4-a_2a_5\geqslant0$,\\
$\Delta_4(f)=(a_3a_4-a_2a_5)(a_1a_2-a_0a_3) - (a_1a_4-a_0a_5)^2 \geqslant 0$,\\
and $\GCD(f_e,f_o)$ has only negative zeros,
\item[(iii)] \begin{equation*}
A \in (0,1], \; B \in \left(0,\tfrac{1}{4}\right],\;
C \in \left(0,\tfrac{1}{4}\right], \; A\geqslant B, \; A \geqslant C ,\;
\frac{(A^2-BC)^2}{A(A-B)(A-C)} \leqslant 1,
\end{equation*}
\item[(iv)] 
\begin{equation*}
A \in (0,1], \; B \in \left(0,\tfrac{1}{4}\right],\;
C \in \left(0,\tfrac{1}{4}\right], \; A \geqslant B, \; A \geqslant C ,\;
A \in [t_1(B,C),s_1(B,C)],
\end{equation*}
\end{itemize}
where 
\begin{equation*}
t_1(u,v)=\max \left\{ \tfrac{1}{4}(1+\sqrt{1-4u})(1-\sqrt{1-4v}) ,
\tfrac{1}{4}(1-\sqrt{1-4u})(1+\sqrt{1-4v}) \right\},
\end{equation*}
\begin{equation*}
s_1(u,v)=\tfrac{1}{4}(1+\sqrt{1-4u})(1+\sqrt{1-4v}).
\end{equation*}
As a special case, the following equivalences also hold:
\begin{itemize}
\item[(i)] $f \in \HH_5$,
\item[(ii)] $\Delta_2(f)=a_3a_4-a_2a_5>0$,\\
$\Delta_4(f)=(a_3a_4-a_2a_5)(a_1a_2-a_0a_3) - (a_1a_4-a_0a_5)^2 > 0$,
\item[(iii)] 
\begin{equation*}
A \in (0,1), \; B \in \left(0,\tfrac{1}{4}\right),\;
C \in \left(0,\tfrac{1}{4}\right), \; A> B, \; A > C ,\;
\frac{(A^2-BC)^2}{A(A-B)(A-C)} < 1,
\end{equation*}
\item[(iv)] 
\begin{equation*}
A \in (0,1), \; B \in \left(0,\tfrac{1}{4}\right),\;
C \in \left(0,\tfrac{1}{4}\right), \; A > B, \; A > C ,\;
A \in (t_1(B,C),s_1(B,C)).
\end{equation*}
\end{itemize}
\end{lemma}

\begin{lemma}\label{rownowaznosc2}
For $g \in \RR_5^+$ as in \eqref{gg} let us define
\[ X:=\frac{b_1 b_4}{b_2 b_3},\ \ \ \  Y:=\frac{b_1 b_5}{b_3^2},
   \ \ \ \  Z:=\frac{b_0 b_4}{b_2^2},\]
Then the following conditions are equivalent:
\begin{itemize}
\item[(i)] $\frac{g\ast Q^3_1}{x} \in \HH_3^\star $, $g*Q^5 \in \HH_5^\star$,
\item[(ii)] $\Delta_2\left(\frac{g* Q^3_1}{x}\right)=b_2b_3-b_1b_4\geqslant 0$,\\
$\Delta_2(g*Q^5)=2(b_3b_4-b_2b_5) \geqslant 0$,\\
$\Delta_4(g*Q^5)=4(b_3 b_4 - b_2b_5)(b_1b_2 - b_0 b_3)
- (b_1 b_4 - b_0 b_5)^2 \geqslant 0$,
\item[(iii)] 
\begin{equation*}
 X \in (0,1], \; Y \in (0,1],\; Z \in (0,1], \; X \geqslant Y,
\; X \geqslant Z,\; \frac{(X^2-YZ)^2}{X(X-Y)(X-Z)}\leqslant 4,
\end{equation*}
\item[(iv)] 
\begin{equation*}
 X \in (0,1], \; Y \in (0,1],\; Z \in (0,1], \; X\geqslant Y,
\; X \geqslant Z,\;t_4(Y,Z)\leqslant 1,\; X \in [t_4(Y,Z),1],
\end{equation*}
where 
\begin{equation*}
t_4(u,v)=\max \left\{(1+\sqrt{1-u})(1-\sqrt{1-v}),
(1-\sqrt{1-u})(1+\sqrt{1-v}) \right\}.
\end{equation*}
\end{itemize}
\end{lemma}

\begin{proof}[Proof of Lemmas 1 and 2]
Let us consider $f, g \in \RR_5^+$ given by \eqref{ff}and \eqref{gg},
respectively. The implication $(i)\Rightarrow(ii)$ is an immediate consequence
of Theorem \ref{agt}. To prove the reverse implication $(i)\Leftarrow(ii)$, we
notice that since $b_2b_3\geqslant b_1b_4$, Theorem \ref{agt} implies that the
polynomial $\frac{g\ast Q^3_1}{x}$ is quasi-stable. For the polynomial
$g*Q^5$, the additional condition $b_2b_3-b_1b_4\geqslant0$, together with the
conditions $b_3b_4-b_2b_5\geqslant 0$ and $b_1b_2-b_0b_3 \geqslant0$, ensures
that the zeros of its even and odd parts are real and negative. Consequently,
$\GCD((g*Q^5)_e,(g*Q^5)_o)$ can only have negative zeros. We distinguish three
cases.
\begin{itemize}
\item[(1)] If $\Delta_2(g*Q^5)> 0$ and $\Delta_4(g*Q^5)>0$, then, by the Liénard-Chipart Theorem, $g*Q^5$ is stable.
\item[(2)]If $\Delta_2(g*Q^5)=0$ and $\Delta_4(g*Q^5)=0$, then one easily
verifies that $\Delta_3(g*Q^5)=0$. Hence by Theorem~\ref{agt}, $g*Q^5$ is
quasi-stable.
\item[(3)]Assume that $\Delta_2(g*Q^5)=0$ and $\Delta_4(g*Q^5)>0$. We show
that $(g*Q^5)_o \prec (g*Q^5)_e$. Define $G_o(x):=
\tfrac{(g*Q^5)_o}{b_5}$ and $G_e(x):=\tfrac{(g*Q^5)_e}{b_4}$. Clearly
$(g*Q^5)_o \prec (g*Q^5)_e$ if and only if $G_o \prec G_e$. This, however,
is equivalent to
\begin{equation*}
\frac{b_1}{b_5}\geqslant\frac{b_0}{b_4} \; \textrm{ and } \;
\exists x_0 \in \RR : G_o(x_0)=G_e(x_0)\leqslant 0 .
\end{equation*}
The first inequality is clear. Solving the equation $G_o(x_0)=G_e(x_0)$ gives
\begin{equation*}
x_0=\frac{b_0b_5-b_1b_4}{2(b_3b_4-b_2b_5)}.
\end{equation*}
Substituting $x_0$ yields
\begin{equation*}
G_o(x_0)=-\frac{\Delta_4(g*Q^5)}{4(b_3b_4-b_2b_5)}=0.
\end{equation*}
Therefore, by the Hermite-Biehler Theorem $g*Q^5$ is quasi-stable.
\end{itemize}
The same reasoning applies to the polynomial $f$.

To show that $(ii)\Leftrightarrow(iii)$, we observe that the condition
$b_2b_3-b_1b_4 \geqslant 0$ is equivalent to $X \in (0,1]$. From the
remaining conditions, we can also easily obtain
\begin{equation*}
X\geqslant Y,\;X \geqslant Z.
\end{equation*}
Additionally, by the Hermite–Biehler Theorem, we know that the polynomials
$(g*Q_5)_e(x)=b_4x^2+2b_2x+b_0$ and $(g*Q_5)_o(x)=b_5x^2+2b_3x+b_1$ have
real zeros, from which we deduce
\begin{equation*}
Y \in \left(0,1\right],\; Z \in \left(0,1\right].
\end{equation*}
To prove the condition
\begin{equation*}
\frac{(X^2-YZ)^2}{X(X-Y)(X-Z)}\leqslant 
\end{equation*}
it suffices to note the equivalence of
\begin{equation*}
4(b_3b_4-b_2b_5)(b_1b_2-b_0b_3)\geqslant(b_1b_4-b_0b_5)^2
\end{equation*}
and
\begin{equation*}
4\frac{b_3^2b_2}{b_1}\left(\frac{b_1b_4}{b_2b_3}-\frac{b_1b_5}{b_3^2}\right)
\frac{b_2^2b_3}{b_4}\left(\frac{b_1b_4}{b_2b_3}-\frac{b_0b_4}{b_2^2}\right)
\geqslant\frac{b_2^4b_3^4}{b_1^2b_4^2}\left(\frac{b_1^2b_4^2}{b_2^2b_3^2}
-\frac{b_1b_5}{b_3^2}\frac{b_0b_4}{b_2^2}\right)^2.
\end{equation*}
Moreover, if  $X=Y$, then $X=Z$ and the above inequality is satisfied.
It suffices to apply the same reasoning to the polynomial $f$.

To verify that $(iii)\Leftrightarrow(iv)$ it is enough to show the equivalences
\begin{itemize}
\item[(1)]  \begin{equation*}
\frac{(A^2-BC)^2}{A(A-B)(A-C)} \leqslant 1 \Leftrightarrow A \in[t_1(B,C),s_1(B,C)],
\end{equation*}
\item[(2)] \begin{equation*}
\frac{(X^2-YZ)^2}{X(X-Y)(X-Z)}\leqslant 4 \Leftrightarrow t_4(Y,Z) \leqslant 1 \textrm{ and } X \in[t_4(Y,Z),1]. 
\end{equation*}
\end{itemize}
Let us consider the function
\begin{equation*}
F(t)=\frac{(t^2-uv)}{t(t-u)(t-v)} \textrm{ for } t\in (0,1],t>u,t>v.
\end{equation*}
Then
\begin{multline*}
F\left(\tfrac{1}{4}(1-\sqrt{1-4u})(1-\sqrt{1-4v})\right)
 =F\left(\tfrac{1}{4}(1+\sqrt{1-4u})(1-\sqrt{1-4v})\right)=\\
=F\left(\tfrac{1}{4}(1-\sqrt{1-4u})(1+\sqrt{1-4v})\right)
 =F\left(\tfrac{1}{4}(1+\sqrt{1-4u})(1+\sqrt{1-4v})\right)=1.
\end{multline*}
Let
\begin{equation*}
t_1(u,v)=\max \left\{ \tfrac{1}{4}(1+\sqrt{1-4u})(1-\sqrt{1-4v}) ,
\tfrac{1}{4}(1-\sqrt{1-4u})(1+\sqrt{1-4v}) \right\},
\end{equation*}
and
\begin{equation*}
s_1(u,v)=\tfrac{1}{4}(1+\sqrt{1-4v})(1+\sqrt{1-4u}).
\end{equation*}
We see that the inequality $\frac{(A^2-BC)^2}{A(A-B)(A-C)}\leqslant 1$ for
$A\geqslant B$ and $A \geqslant C$ holds if and only if
$A \in[t_1(B,C),s_1(B,C)]$.
In the same way, we verify that
\begin{multline*}
F\left((1-\sqrt{1-u})(1-\sqrt{1-v})\right)
 =F\left((1+\sqrt{1-u})(1-\sqrt{1-v})\right)=\\
=F\left((1-\sqrt{1-u})(1+\sqrt{1-v})\right)
 =F\left((1+\sqrt{1-u})(1+\sqrt{1-v})\right)=4.
\end{multline*}
Let
\begin{equation*}
t_4(u,v)=\max \left\{(1+\sqrt{1-u})(1-\sqrt{1-v}),
(1-\sqrt{1-u})(1+\sqrt{1-v}) \right\},
\end{equation*}
and
\begin{equation*}
s_4(u,v)=(1+\sqrt{1-u})(1+\sqrt{1-v}).
\end{equation*}
Similarly, we see that since $s_4(u,v)\geqslant 1$, the inequality
$\frac{(X^2-YZ)^2}{X(X-Y)(X-Z)} \leqslant 4$ for $X \geqslant Y$,
$X \geqslant Z$ holds if and only if $t_4(Y,Z)\leqslant1 \textrm{ and } X \in[t_4(Y,Z),1]$.
\end{proof}

\begin{lemma}\label{funkcje}
Let us consider the following functions
\begin{equation*}
\varphi_-(t)=1-\sqrt{1-t}, \qquad \varphi_+(t)=1+\sqrt{1-t}.
\end{equation*}
Then, for all $a \in(0,1)$ the functions
 \begin{equation*}
 t \mapsto \frac{\varphi_-(at)}{\varphi_-(t)}, \qquad
   t \mapsto \frac{\varphi_+(at)}{\varphi_-(t)}
   \end{equation*}
are decreasing for $t \in (0,1]$, and the functions
\begin{equation*}
t \mapsto  \frac{\varphi_-(at)}{\varphi_+(t)}, \qquad
  t \mapsto \frac{\varphi_+(at)}{\varphi_+(t)}
  \end{equation*}
are increasing for $t \in (0,1]$
\end{lemma}
\begin{proof}
Let us set $a \in (0,1)$ and consider
\begin{equation*}
\varPhi(a,t)=\frac{\varphi_-(at)}{\varphi_-(t)}=
\frac{1-\sqrt{1-at}}{1-\sqrt{1-t}}.
\end{equation*}
By calculating the derivative
\begin{equation*}
\frac{\partial \varPhi}{\partial t}(a,t)=
\frac{a(\sqrt{1-t}-1)-\sqrt{1-at}+1}{2(1-\sqrt{1-t})^2\sqrt{1-t}\sqrt{1-at}}.
\end{equation*}
It is enough to show that
\begin{equation*}
a(\sqrt{1-t}-1)-\sqrt{1-at}+1<0 \qquad \forall t \in (0,1].
\end{equation*}
We have
\begin{equation*}
1-\sqrt{1-at}<a(1-\sqrt{1-t})
\end{equation*}
which, after applying elementary operations (bearing in mind that $a \in (0,1)$)
is equivalent to
\begin{equation*}
0<(a^2+a+1)t^2-3(1+a)t+3.
\end{equation*}
It is easily verified that for $t \in (0,1]$ the inequality holds. The
remaining cases can be proved in a similar way.
\end{proof}

\begin{proof}[Proof of Theorem \ref{main5}]
In light of Corollary \ref{necessary5} and Theorem \ref{main4}, we only
need to prove the implication in one direction for the case where $m=5$. Let
us consider $f, g \in \RR_5^+$ satisfying the conditions of the theorem, given
by \eqref{ff} and \eqref{gg} respectively. Thanks to Lemma \ref{rownowaznosc1}
and Lemma \ref{rownowaznosc2} (using the same notation), it suffices to show
that from
\begin{equation*}
A \in (0,1], \; B \in \left(0,\tfrac{1}{4}\right],\;
C \in \left(0,\tfrac{1}{4}\right], \; A \geqslant B, \; A \geqslant C ,\;
A \in [t_1(B,C),s_1(B,C)],
\end{equation*}
where \begin{equation*}
t_1(u,v)=\max \left\{ \tfrac{1}{4}(1+\sqrt{1-4u})(1-\sqrt{1-4v}) ,
\tfrac{1}{4}(1-\sqrt{1-4u})(1+\sqrt{1-4v}) \right\},
\end{equation*}
\begin{equation*}
s_1(u,v)=\tfrac{1}{4}(1+\sqrt{1-4u})(1+\sqrt{1-4v}),
\end{equation*}
and
\begin{equation*}
 X \in (0,1], \; Y \in (0,1],\; Z \in (0,1], \; X \geqslant Y,
\; X \geqslant Z,\;t_4(Y,Z)\leqslant 1,\; X \in [t_4(Y,Z),1],
\end{equation*}
where 
\begin{equation*}
t_4(u,v)=\max \left\{(1+\sqrt{1-u})(1-\sqrt{1-v}),
(1-\sqrt{1-u})(1+\sqrt{1-v}) \right\},
\end{equation*}
it follows that
\begin{equation*}
AX \in (0,1], \; BY \in \left(0,\tfrac{1}{4}\right],\;
CZ \in \left(0,\tfrac{1}{4}\right], \; AX \geqslant BY, \; AX \geqslant CZ ,\;
AX \in [t_1(BY,CZ),s_1(BY,CZ)].
\end{equation*}
The only non-obvious condition is
\begin{equation*}
AX \in [t_1(BY,CZ),s_1(BY,CZ)].
\end{equation*}
Note that if $A \leqslant s_1(B,C)$, then since $X \leqslant 1$ we get
\begin{equation*}
AX \leqslant s_1(B,C)\leqslant s_1(BY,CZ)
\end{equation*}
because $Y\leqslant 1,Z \leqslant 1$.
It therefore remains to prove that
\begin{equation*}
t_1(BY,CZ) \leqslant AX.
\end{equation*}
It suffices to show that
\begin{equation*}
t_1(BY,CZ)\leqslant t_1(B,C)\cdot t_4(Y,Z).
\end{equation*}
We consider three cases:
\begin{itemize}
\item[(1)] Let us assume that $B\geqslant C,Y \geqslant Z$. Then
\begin{align*}
t_1(B,C)&=\tfrac{1}{4}(1-\sqrt{1-4B})(1+\sqrt{1-4C})\\
t_4(Y,Z)&=(1-\sqrt{1-Y})(1+\sqrt{1-Z})\\
t_1(BY,CZ)&=\tfrac{1}{4}(1-\sqrt{1-4BY})(1+\sqrt{1-4CZ}).
\end{align*}
In order to show that
\begin{equation*}
(1-\sqrt{1-4BY})(1+\sqrt{1-4CZ})\leqslant
(1-\sqrt{1-4B})(1+\sqrt{1-4C})(1-\sqrt{1-Y})(1+\sqrt{1-Z})
\end{equation*}
by rearranging the inequality, we obtain
\begin{equation*}
\frac{1-\sqrt{1-4BY}}{(1-\sqrt{1-4B})(1-\sqrt{1-Y})}\leqslant
\frac{1-\sqrt{1-4CZ}}{(1-\sqrt{1-4C})(1-\sqrt{1-Z})}.
\end{equation*}
From Lemma \ref{funkcje}
\begin{equation*}
\frac{1-\sqrt{1-4BY}}{(1-\sqrt{1-4B})(1-\sqrt{1-Y})}\leqslant
\frac{1-\sqrt{1-4CY}}{(1-\sqrt{1-4C})(1-\sqrt{1-Y})} \leqslant
\frac{1-\sqrt{1-4CZ}}{(1-\sqrt{1-4C})(1-\sqrt{1-Z})}.
\end{equation*}

\item[(2)] Let us assume that $B \leqslant C,Y \geqslant Z$ and $BY \geqslant CZ$.
Then
\begin{align*}
t_1(B,C)&=\tfrac{1}{4}(1+\sqrt{1-4B})(1-\sqrt{1-4C})\\
t_4(Y,Z)&=(1-\sqrt{1-Y})(1+\sqrt{1-Z})\\
t_1(BY,CZ)&=\tfrac{1}{4}(1-\sqrt{1-4BY})(1+\sqrt{1-4CZ}).
\end{align*}
In order to show that
\begin{equation*}
(1-\sqrt{1-4BY})(1+\sqrt{1-4CZ}) \leqslant
(1+\sqrt{1-4B})(1-\sqrt{1-4C})(1-\sqrt{1-Y})(1+\sqrt{1-Z})
\end{equation*}
by rearranging the inequality, we obtain
\begin{equation*}
\frac{1-\sqrt{1-4BY}}{(1+\sqrt{1-4B})(1-\sqrt{1-Y})} \leqslant
\frac{1-\sqrt{1-4CZ}}{(1+\sqrt{1-4C})(1-\sqrt{1-Z})}.
\end{equation*}
From Lemma \ref{funkcje}
\begin{equation*}
\frac{1-\sqrt{1-4BY}}{(1+\sqrt{1-4B})(1-\sqrt{1-Y})}\leqslant
\frac{1-\sqrt{1-4CY}}{(1+\sqrt{1-4C})(1-\sqrt{1-Y})} \leqslant
\frac{1-\sqrt{1-4CZ}}{(1+\sqrt{1-4C})(1-\sqrt{1-Z})}.
\end{equation*}

\item[(3)] Let us assume that $B \leqslant C,Y \geqslant Z$ and $BY \leqslant CZ$.
Then
\begin{align*}
t_1(B,C)&=\tfrac{1}{4}(1+\sqrt{1-4B})(1-\sqrt{1-4C})\\
t_4(Y,Z)&=(1-\sqrt{1-Y})(1+\sqrt{1-Z})\\
t_1(BY,CZ)&=\tfrac{1}{4}(1+\sqrt{1-4BY})(1-\sqrt{1-4CZ}).
\end{align*}
In order to show that
\begin{equation*}
(1+\sqrt{1-4BY})(1-\sqrt{1-4CZ}) \leqslant
(1+\sqrt{1-4B})(1-\sqrt{1-4C})(1-\sqrt{1-Y})(1+\sqrt{1-Z})
\end{equation*}
by rearranging the inequality, we obtain
\begin{equation*}
\frac{1+\sqrt{1-4BY}}{(1+\sqrt{1-4B})(1-\sqrt{1-Y})} \leqslant
\frac{1+\sqrt{1-4CZ}}{(1+\sqrt{1-4C})(1-\sqrt{1-Z})}.
\end{equation*}
From Lemma \ref{funkcje}
\begin{equation*}
\frac{1+\sqrt{1-4BY}}{(1+\sqrt{1-4B})(1-\sqrt{1-Y})} \leqslant
\frac{1+\sqrt{1-4CY}}{(1+\sqrt{1-4C})(1-\sqrt{1-Y})} \leqslant
\frac{1+\sqrt{1-4CZ}}{(1+\sqrt{1-4C})(1-\sqrt{1-Z})},
\end{equation*}
which proves (ii).
\end{itemize}
In case (i), when $f \in \HH_5$, the proof proceeds in the same way, with the
exception that the relevant strict inequalities hold.
\end{proof}

\section{The idealizer of the set of quasi-stable polynomials}\label{quasi}

One might wonder about the idealizer for the set $\HH_n^\star$. Again for each
$n$ we seek a maximal subsemigroup $X_n^\star$ with the Hadamard product $\ast$,
such that $\HH_n^\star \subset X_n^\star$, $\II_n \in X_n^\star$, and for
$m \leqslant\ n$
\begin{equation*}
\ast : X_n^\star \times \HH_m^\star \rightarrow \HH_m^\star.
\end{equation*}
Since polynomials in $\HH_n^\star$ may have vanishing coefficients, and
$\HH_n^\star$ has to be a subset of $X_n^\star$, we must allow the polynomials
from $X_n^\star$ to also have nonnegative coefficients. Hence
\begin{equation*}
X_n^\star = \left\{ g(x)= \sum_{i=0}^n b_i x^i,\; b_i \geqslant 0
\;(i=1,\ldots,n-1),\; b_0,b_n >0 : \forall m \leqslant n, \;\forall f \in
\HH_m^\star \;\;\; f \ast g \in \HH_m^\star \right\}
\end{equation*}

It is easy to observe that for odd $n$, we have $X_n = X_n^\star$. Indeed, in
this case, the quasi-stable polynomials under consideration have all coefficients
positive, and hence (provided that $b_n, b_0 > 0$) all coefficients of any
$g \in X_n^\star$ must also be positive.
In contrast, when $n$ is even, quasi-stable polynomials may have vanishing
odd-indexed coefficients. In such cases, they are of the form
$f(x) = f_e(x^2)$, where $f_e$ is a polynomial with only real zeros
(by the Hermite–Biehler theorem). To obtain $X_n^\star$, it is therefore
sufficient to enlarge the set $X_n$ by including polynomials of the form
$g(x) = g_e(x^2)$, where the coefficients of $g_e$, when combined with those
of $f_e$ via the Hadamard product, preserve the reality of zeros.
Let $\Hyp_n^+$ denote the class of polynomials of degree $n$ with only real
negative zeros. Then
\begin{equation*}
X_{2l}^\star = X_{2l} \cup \left\{ g(x) = g_e(x^2),\; g_e \in \RR_l^+ :
\forall f \in \Hyp_l^+ \; g_e * f \in \Hyp_l^+ \right\}.
\end{equation*}
In other words, the coefficients of $g_e$ must form a \textbf{finite multiplier
sequence} (cf.\ \cite{k2008}) preserving $\Hyp_l^+$. Thus, in the quasi-stable
setting, we incorporate a special class of polynomials that highlights the
strong connections between the results of this paper and the Pólya-Schur
Theorem. Thus
\begin{multline*}
X_4^\star=X_4 \cup \left\{g(x)=g_e(x^2),\; g_e \in \RR_2^+ :
g*Q_4 \in \HH_4^\star \right\} =\\
= X_4 \cup \left\{g(x)=g_e(x^2),\; g_e \in \RR_2^+ :
g_e(x)*(x+1)^2 \in \Hyp_2^+ \right\}.
\end{multline*}
In analogy with the previous construction, we set $Y_n^\star=Y_n$ for odd $n$,
while for $n=2l$ we define $Y_{2l}^\star$ as follows
\begin{multline*}
Y_{2l}^\star:= Y_{2l} \cup \left\{g(x)=g_e(x^2),\; g_e \in \RR_l^+ :
g*Q_{2\nu} \in \HH_{2\nu}^\star  \;\; \forall\; \nu \in \{2,\ldots,l\}
\right\}=\\
=Y_{2l} \cup \left\{g(x)=g_e(x^2),\; g_e \in \RR_l^+ :
g_e(x)*(x+1)^\nu \in \Hyp_\nu^+  \;\; \forall\;
\nu \in \{2,\ldots,l\}\right\}.
\end{multline*}
We therefore formulate the corresponding conjecture for quasi-stable polynomials.
\begin{conjecture}\label{conjecture2}
For any $n \in \NN$, equality $Y_n^\star = X_n^\star$ holds.
\end{conjecture}

\section{A special case}\label{misc}

We now show a sufficient condition for a special case for all odd $n$, which
provides the additional rationale behind our hypothesis.

\begin{theorem}\label{scase}
Suppose $G(x)=g(x^2)+xg(x^2)=(x+1)g(x^2)$, where $g \in \RR_k^+$. Then
the following implication holds
\begin{equation*}
G \ast Q^{2k+1} \in \HH_{2k+1}^\star \Rightarrow F \ast G \in
\HH_{2k+1}^\star \qquad \forall F \in \HH_{2k+1}^\star.
\end{equation*}
\end{theorem}

First, we require the following auxiliary corollary and theorem
(cf.\ \cite[Cor.~5.5.8, Cor.~5.5.10, Thm.~6.3.8]{rs2002}).

\begin{corollary}\label{RSch}
Let $f(x)=\sum_{\nu=0}^p a_\nu x^\nu$ be a polynomial of degree $p \geqslant 2$
with only real zeros. Let $g(x)=\sum_{\nu=0}^q b_\nu x^\nu$ be a polynomial of
degree $q\geqslant 2$, whose non-vanishing zeros are all real and of the same
sign. For both polynomials, a possible zero at the origin is assumed to be of
order less than $m:= \min \{p,q\}$. Then, for any integers $k \geqslant -1$,
$l \geqslant -1$, and $n \geqslant \max \{p,q\}$, the polynomial
\begin{equation*}
h_{kl}:=\sum_{\mu=0}^m \frac{a_\mu b_\mu}{(\mu !)^k ((n-\mu)!)^l} x^\mu
\end{equation*}
has only real zeros.
As a special case, if the zeros of $f$ are all of the same sign $\sigma$, and
those of $g$ are all of the same sign $\tau$, then, for any integers
$k\geqslant -1, l \geqslant -1$, and $n \geqslant \max \{p,q\}$, the polynomial
$h_{kl}$ has only real zeros, all of them being of sign $-\sigma \tau$.
\end{corollary}

\begin{theorem}[Hermite-Kakeya]\label{HK}
Let $P$ and $Q$ be non-constant polynomials with real coefficients. Then $P$
and $Q$ have strictly interlacing zeros if and only if, for all $\lambda,
\mu \in \RR$ such that $\lambda^2+\mu^2 \neq 0$, the polynomial
$R(x):= \lambda P(x)+\mu Q(x)$ has simple, real zeros.
\end{theorem}

\begin{proof}[Proof of Theorem \ref{scase}]
Let us consider a polynomial $G$ that satisfies the conditions of the theorem.
When $k=1$, then $G$ is quasi-stable. We consider $k \geqslant 2$.
Note that the product $G \ast Q^{2k+1}$ can be written as
\begin{equation*}
G \ast Q^{2k+1} (x) = ((x+1)g(x^2)) * ((x+1)(x^2+1)^k)=
(x+1)(g(x^2) * (x^2+1)^k).
\end{equation*}
Since $G \ast Q^{2k+1}$ is quasi-stable, by the Hermite–Biehler Theorem,
$g(x)*(x+1)^k$ has only negative zeros. Let $F(x)=f_e(x^2)+xf_o(x^2)$.
If $F \in \HH_{2k+1}^\star$, then, again by the Hermite–Biehler Theorem,
$f_o \prec f_e \textrm{ and } \GCD(f_e,f_o)=w(x) \textrm{ has only negative
zeros. }$ Let $\tilde{f_o}(x):=\frac{f_o(x)}{w(x)} $ and
$\tilde{f_e}(x):=\frac{f_e(x)}{w(x)}$.
By applying Theorem \ref{HK}, we know that the polynomial
\begin{equation*}
H(x)=\lambda f_o(x) + \mu f_e(x)=w(x)\left(\lambda \tilde{f_o}(x)
+ \mu \tilde{f_e}(x)\right)
\end{equation*}
has only real zeros for all $\lambda, \mu \in \RR  \textrm{ such that }
\lambda^2+\mu^2 \neq 0$ as
\begin{equation*}
\tilde{H}(x)=\lambda \tilde{f_o}(x) + \mu \tilde{f_e}(x)
\end{equation*}
has only simple, real zeros for all $\lambda, \mu \in \RR  \textrm{ such that }
\lambda^2+\mu^2 \neq 0$.
In turn, from Corollary \ref{RSch}, we know that since $g * (x+1)^k$ has
negative roots, the polynomial
\begin{equation*}
\left(H * g\right) (x) = (\lambda f_o(x) + \mu f_e(x)) * g(x)
= \lambda f_o*g(x) + \mu f_e*g(x)
\end{equation*}
has only real zeros for all $\lambda, \mu \in \RR  \textrm{ such that }
\lambda^2+\mu^2 \neq 0$. Let us denote the common multiple zeros of $f_o*g$
and $f_e*g$ as $\xi_i \;(i=0, \ldots, m)$, where $m \leqslant k$. If $m=k$,
then $f_o*g=f_e*g$, which have negative zeros (from Corollary \ref{RSch}), and
the polynomial $F*G$ is quasi-stable. If, however, $m<k$ then
\begin{equation*}
(H * g)(x)=\sum_{i=0}^m(x-\xi_i)\left(\lambda
\frac{f_o*g(x)}{\sum_{i=0}^m(x-\xi_i)} + \mu
\frac{f_e*g(x)}{\sum_{i=0}^m(x-\xi_i)}\right).
\end{equation*}
Thus, we obtain that the polynomial
\begin{equation*}
\lambda \frac{f_o*g(x)}{\sum_{i=0}^m(x-\xi_i)} + \mu
\frac{f_e*g(x)}{\sum_{i=0}^m(x-\xi_i)}
\end{equation*}
has simple, real zeros for all $\lambda, \mu \in \RR  \textrm{ such that }
\lambda^2+\mu^2 \neq 0$. We conclude that
\begin{equation*}
 \frac{f_o*g(x)}{\sum_{i=0}^m(x-\xi_i)} \prec
\frac{f_e*g(x)}{\sum_{i=0}^m(x-\xi_i)},
\end{equation*}
and therefore that
\begin{equation*}
f_o*g \prec f_e *g.
\end{equation*}
Again, by Corollary \ref{RSch}, the zeros of $f_o*g$ and $f_e*g$ are negative,
which completes the proof.
\end{proof}

Unfortunately, to the best of our knowledge, in this case, there are no
available tools for determining the stability of $F*G$, even under the
assumption that $F$ is stable.

\section{Examples}\label{exes}

\begin{ex}\label{ex1}
If $f \in \HH_5$ and $g \in \WW_n$, then $f*g$ need not be stable.
Let us consider
\begin{equation*}
f(x)=100x^5 + 230x^4 + 80x^3 + 164x^2 + 8x + 16
\end{equation*}
and
\begin{equation*}
g(x)=6.17x^5 + 6.4x^4 + 8.96x^3 + 6.62x^2 + 6.4x + 4.66
\end{equation*}
It is easy to verify that $g \in \WW_n$ and that $f$ is stable, as
\begin{equation*}
\Delta_2(f)=2000, \qquad \Delta_4(f)=6400.
\end{equation*}
Their Hadamard product is
\begin{equation*}
(f \ast g) (x) = 617 x^5 +1472 x^4 +716.8x^3+1085.68 x^2 + 51.2 x
+ 74.56.
\end{equation*}
Checking the principal minors of the matrix $H(f \ast g)$, we have
\begin{equation*}
\Delta_2(f \ast g) = 385265.04 
\end{equation*}
but
\begin{equation*}
\Delta_4(f \ast g) = -3.686087108608  \cdot  10^7.
\end{equation*}
In fact, $f \ast g$ has two roots with positive real parts
$0.000062127 \pm 0.276826i$.
\end{ex}

\begin{ex}\label{ex2}
As for $g \in \WW_n$ required inequalities are equivalent to positivity of
every two by two minor of a matrix $H(g)$, we could also ask if for $n=5$ the
positivity of every three by three minor is required as a sufficient condition.
Even though it might be, it is not necessary. Consider
$g(x)=x^5 + x^4 + 5.5 x^3 + 4.75 x^2 + 10 x + 4.5$. Its Hurwitz matrix is
as follows
\begin{equation*}
H(g)=\begin{bmatrix}
1 & 4.75& 4.5 & 0 & 0 \\
1 & 5.5 & 10  & 0 & 0\\
0 & 1   & 4.75& 4.5 & 0 \\
0 & 1   & 5.5 & 10  & 0 \\
0 & 0   & 1 & 4.75& 4.5  \\
\end{bmatrix}.
\end{equation*}
The polynomial $g$ satisfies both conditions $\frac{g\ast Q^3_1}{x} \in \HH_3$
and $g\ast Q^5 \in \HH_5$, however, its third principle minor is negative
\begin{equation*}
\begin{vmatrix}
1 & 4.75& 4.5   \\
1 & 5.5 & 10   \\
0 & 1   & 4.75  \\
\end{vmatrix} =-1.9375.
\end{equation*}
The middle one, on the other hand, is positive
\begin{equation*}
\begin{vmatrix}
 5.5 & 10  & 0 \\
 1   & 4.75 & 4.5\\
 1 & 5.5 & 10 \\
\end{vmatrix} =70.125.
\end{equation*}
\end{ex}

\section*{Acknowledgments}
The author thanks Michał Góra for having brought the problem of this paper to his attention,


\bibliographystyle{cas-model2-names}

\end{document}